\documentclass[12pt]{article}
\usepackage{amsmath, amssymb, amscd, amsthm, amsfonts, xfrac}
\numberwithin{equation}{section}
\usepackage{graphicx}
\usepackage{hyperref}
\usepackage{authblk}
\usepackage{url}
\hypersetup{backref, colorlinks=true}
\newtheorem{theorem}{Theorem}[section]

\newtheorem{exam}[theorem]{Example}

\newcommand{\sq}{\sqrt[n]{a_n}}
\def\les{\leqslant}
\def\ges{\geqslant}

\title{\textbf{Log-behavior of the root sequences of P-recursive sequences}}
\author{Qing-hu Hou and Zhongjie Li}
\affil{School of Mathematics \break Tianjin University \break Tianjin 300072, China \break\texttt{qh\_hou@tju.edu.cn, lizhongjie@tju.edu.cn}}
\date{}

\begin{document}
\maketitle
\begin{abstract}
In recent years, Sun has proposed numerous conjectures regarding the log-concavity of root sequences $\{\sq\}_{n\geqslant 1}$. We establish criteria for the asymptotic log-concavity of $\{\sq\}_{n\geqslant 1}$ and the asymptotic ratio log-convexity of $\{\sq\}_{n\geqslant 1}$ for $P$-recursive sequences $\{a_n\}_{n\geqslant{0}}$. Additionally, by the aid of symbolic computation, we present a systematic approach to determine the explicit integer $N$ such that the sequence $\{\sq\}_{n\geqslant{N}}$ is log-concave and the sequence $\{\sq\}_{n\geqslant N}$ is ratio log-convex.
\end{abstract}


\section{Introduction}
A sequence $\{a_{n}\}_{n\geqslant 1}$  is called \textit{log-convex} ({\it log-concave}, respectively) if the inequality $a_{n}^2  \leqslant a_{n+1} a_{n-1}$ ($a_{n}^2  \ges a_{n+1} a_{n-1}$, respectively) holds for all $n \ges 2$. For a non-negative sequence $\{a_n\}_{n \ges 1}$, we define its \emph{root sequence} as $\{\sqrt[n]{a_n}\}_{n \ges 1}$. Recently, Sun~\cite{sun2013conjectures, sun2021new} proposed several conjectures regarding the log-concavity of the root sequences of combinatorial sequences. Some of these conjectures have been confirmed, as shown in \cite{hou2012monotonicity, sun2018ratio, sun2020log, zhang2021proof}. In particular, Xia~\cite{xia2018log} developed a systematic method to prove the log-concavity of $\{\sq\}_{n\geqslant 1}$ when $a_n$ satisfies a second-order linear reucrrence.

We consider the general case where $\{a_n\}_{n \ges 1}$ is a \emph{P-recursive sequence}. Recall that a $P$-recursive sequence $\{a_n\}_{n \ges 1}$  of order $d$ satisfies a recurrence relation of the form
\[
p_0(n) a_n + p_1(n) a_{n+1} + \cdots + p_d(n) a_{n+d} = 0, \quad \forall\, n \ges 1,
\]
where $p_{i}(n)$ are polynomials in $n$. Wimp and Zeilberger \cite{wimp1985resurrecting} (see also \cite[Sec. VIII.7]{Ana}) showed that a $P$-recursive sequence $\{a_{n}\}_{n\ges 1}$ is asymptotically equal to a linear combination of terms in the form of
\begin{equation}\label{eq-Qs}
e^{Q(\rho,n)}s(\rho,n), 
\end{equation}
where
\begin{align}
Q(\rho,n) & = \mu_{0}n\log{n} + \sum_{j=1}^{\rho}\mu_{j}n^{j/\rho}, \label{eq-Q}\\
s(\rho,n) & = n^r\sum_{j=0}^{t-1}(\log{n})^j\sum_{s=0}^{\infty}b_{sj}n^{-s/\rho}, \label{eq-s}
\end{align}
with $\rho,\ t$ being positive integers and $\mu_{j},\ r,\ b_{sj}$ being complex numbers. We focus on the case where $a_n$ asymptotically takes the form of \eqref{eq-Qs} and $t=1$ in \eqref{eq-s}, which encompasses the majority of combinatorial sequences.

Based on the asymptotic expression of $a_n$, we establish a criterion for $a_n$ such that the root sequences $\{\sqrt[n]{a_n}\}$ is {\it asymptotically log-concave}, i.e., 
\[
\sqrt[n]{a_n}^2 \ges \sqrt[n-1]{a_{n-1}} \sqrt[n+1]{a_{n+1}}
\] 
holds for all sufficiently large $n$. We define a sequence $\{a_n\}_{n \ges 1}$ to be {\it ratio log-convex} if $\{a_{n+1}/a_n\}_{n \ges 1}$ is log-convex. By a similar discussion, we provide a criterion for $a_n$ such that the root sequence is asymptotically ratio log-convex.

Our next objective is to identify a specific integer $N$ that guarantees the log-concavity of the sequence $\{\sqrt[n]{a_n}\}_{n \ges N}$. We solve the problem by establishing bounds for both $r_n=a_{n+1}/a_n$ and $a_n$. Notably,  
Hou and Zhang \cite{hou2019asymptotic} have developed an algorithm to compute $N$ and provide bounds for the ratio $r_n$ when $n \ges N$. Combing with the asymptotic expression of $a_n$, we are enable to give an upper bound of $a_n$ as intended. Then we can use \texttt{Maple} or \texttt{Mathematica} to find the derivative and limiting value in order to prove the constant positivity or negativity of the rational function.

\indent This paper is organized as follows. In Section 2, by studying the asymptotic expressions of $P$-recursive sequences, we give a criterion of the asymptotic log-concavity for the sequence $\{\sq\}_{n \geqslant 1}$ and the asymptotic ratio log-convexity for $\{\sq\}_{n \geqslant 1}$. In Section 3, by studying the log-behavior of $\{\sq\}_{n \geqslant 1}$ and the monotonicity of $\left\{\frac{\sqrt[n]{a_{n+1}}}{\sq}\right\}_{n \geqslant 1}$, we extend the result of~\cite{hou2019asymptotic} to give a systematic method to find $N$ such that the sequence $\{\sq\}_{n\geqslant{N}}$ is log-concave and $\{\sq\}_{n \geqslant N}$ is ratio log-convex.

\section{The asymptotic log-behavior of $\{\sq\}_{n \ges 1}$}
In this section, we consider that the asymptotic log-concavity and the asymptotic ratio log-convexity. 

Firstly, we give a sufficient condition for the log-concavity of $\{\sq\}_{n \geqslant 1}.$
\begin{theorem}\label{t1}
Suppose that $\{a_n\}_{n\geqslant{1}}$ is a P-recursive sequence whose asymptotic expression is
\begin{align}
a_{n}\sim  e^{Q(\rho,n)}\cdot n^{r}\left(\sum_{s=0}^{M}b_{s}n^{-s/\rho} + o(n^{-M/\rho})\right), \label{eq-1}   
\end{align}
where
\[
Q(\rho,n) = \mu_{0}n\log{n} + \sum_{j=1}^{\rho}\mu_{j}n^{j/\rho},
\]
with $\rho, M$ being positive integers and $\mu_{j}, r, b_{s}$ being real numbers such that 
\[
b_{0} >0 \quad \mbox{and} \quad M/\rho > 2.
\]
Then for the following three cases, the sequence $\{\sqrt[n]{a_n}\}_{n\geqslant{1}}$ is asymptotically log-concave.
\begin{itemize}
	\item[{\rm (1)}]
	$\mu_{0}>{0}$.
	\item[{\rm (2)}]
	 $\mu_{0} = 0$ and $ \mu_{k}<{0},$ where $k = \max\limits_{1\leqslant{j}\leqslant{\rho-1}} \{j \colon \mu_{j}\ne{0}\}$.
	 \item[{\rm (3)}]
	 $\mu_{0} = 0,\ \mu_{j} = 0$ for $1\leqslant{j}\leqslant{\rho-1}$, and $r<{0}$. 
\end{itemize}

\end{theorem}

\begin{proof}
By definition, our objective is to establish the criteria for $a_n$ such that 
\[\sqrt[n+1]{a_{n+1}}\cdot\sqrt[n-1]{a_{n-1}}\leqslant{(\sq)^2}.\]
Equivalently, we need to ensure that
\begin{align}
\Delta = \dfrac{\log {a_{n+1}}}{n+1} + \dfrac{\log {a_{n-1}}}{n-1} - \dfrac{2 \log {a_{n}}}{n} \leqslant{0}. \label{2}
\end{align}

From the asymptotic formula (\ref{eq-1}) for $a_n$, we obtain an asymptotic formula for $(\log {a_{n}})/n$:
\begin{align*}
\dfrac{\log {a_{n}}}{n} &\sim \mu_{0}\log{n} +\sum\limits_{j=1}^{\rho}
\dfrac{\mu_{j}}{n^{1-j/\rho}} +\dfrac{r\log{n}}{n} + \frac{\log \left( \sum\limits_{s=0}^{M} \frac{b_{s}}{n^{s/\rho}}\right)}{n} +o\left(\dfrac{1}{n^{1+ M/\rho}}\right)\\
&= \mu_{0}\log{n} +\sum\limits_{j=1}^{\rho}
\dfrac{\mu_{j}}{n^{1-j/\rho}} +\dfrac{r\log{n}}{n} +\sum\limits_{s=0}^{M} \frac{\tilde{b}_{s}}{n^{1+s/\rho}} +o\left(\dfrac{1}{n^{1+ M/\rho}}\right),
\end{align*}
where $\tilde{b}_0 =\log b_0$ and $\tilde{b}_i$ is a polynomial in $b_1/b_0, \ldots, b_i/b_0$.

Now we will estimate $f_{n+1}+f_{n-1}-2f_{n}$ for $f_n$ representing one of the four terms in the above equation.

For $f_{n}=\mu_{0}\log{n}$, we have
\begin{align}
f_{n+1}+f_{n-1}-2f_{n} &=\mu_{0}\left(\log({n+1})+\log({n-1})-2\log{n}\right)\notag\\
&= \mu_{0}\log\left({1-\frac{1}{n^2}}\right)\notag\\
&= -\frac{\mu_{0}}{n^2}+\dots+o\left(\frac{1}{n^{1+M}}\right).\label{3}
\end{align}
For $f_{n}=\sum\limits_{j=1}^{\rho}
\dfrac{\mu_{j}}{n^{1-j/\rho}}$, we have
\begin{align}
f_{n+1}+f_{n-1}-2f_{n} &=\sum_{j=1}^{\rho}\mu_{j}\left(\frac{1}{(n+1)^{1-j/\rho}}+\frac{1}{(n-1)^{1-j/\rho}}-\frac{2}{n^{1-j/\rho}}\right)\notag\\
&=\sum_{j=1}^{\rho}\frac{\mu_{j}}{n^{1-j/\rho}}\left(\frac{1}{(1+\frac{1}{n})^{1-j/\rho}}+\frac{1}{(1-\frac{1}{n})^{1-j/\rho}}-2\right)\notag\\
&=\sum_{j=1}^{\rho}\mu_{j}\left(\frac{(j/\rho-1)(j/\rho-2)}{n^{3-j/\rho}}+\dots+o\left(\frac{1}{n^{1+M}}\right)\right).\label{4}
\end{align}
For $f_{n}=\dfrac{r\log{n}}{n}$, we have
\begin{align}
f_{n+1}+f_{n-1}-2f_{n}&=r\left(\frac{\log{(n+1)}}{n+1}+\frac{\log{(n-1)}}{n-1}-\frac{2\log{n}}{n}\right)\notag\\
&=r\left(\frac{n(n-1)\log{(1+\frac{1}{n})}+n(n+1)\log{(1-\frac{1}{n})}+2\log{n}}{n(n-1)(n+1)}\right)\notag\\
&=r\left(\frac{2\log{n}-3}{n^3}+\dots+o\left(\frac{1}{n^{1+M}}\right)\right).\label{5}
\end{align}
Finally, for $f_{n}=\sum\limits_{s=0}^M\dfrac{\tilde{b}_s}{n^{1+s/\rho}}$, we have
\begin{align}
f_{n+1}+f_{n-1}-2f_{n}&=\sum\limits_{s=0}^M \tilde{b}_{s}\left(\frac{1}{(n+1)^{1+s/\rho}}+\frac{1}{(n-1)^{1+s/\rho}}-\frac{2}{n^{1+s/\rho}}\right)\notag\\
&=\sum\limits_{s=0}^M \frac{\tilde{b}_{s}}{n^{1+s/ \rho}}\left(\frac{1}{(1+\frac{1}{n})^{1+s/\rho}}+\frac{1}{(1-\frac{1}{n})^{1+s/\rho}}-2\right)\notag\\
&=\sum\limits_{s=0}^M \tilde{b}_{s}\left(\frac{(1+s/\rho)(2+s/\rho)}{n^{3+s/\rho}}+\dots+o\left(\frac{1}{n^{1+M/\rho}}\right)\right).\label{6}
\end{align}

Combining (\ref{3})--(\ref{6}) together,  we arrive at
\begin{align}
\Delta = & -\frac{\mu_{0}}{n^2}+\sum_{j=1}^{\rho-1}\frac{\mu_{j}(j/\rho-1)(j/\rho-2)}{n^{3-j/\rho}}+\frac{r(2 \log{n}-3)}{n^3} \nonumber \\
& +\sum\limits_{s=0}^M\frac{\tilde{b}_{s}(1+s/\rho)(2+s/\rho)}{n^{3+s/\rho}} + \dots + o\left(\frac{1}{n^{1+M/\rho}}\right). \label{delta}
\end{align}
It is straightforward to see that when the parameters satisfy one of (1)--(3), the coefficient of the dominant term in the asymptotic expression of $\Delta$ is negative.
\end{proof}

{\noindent \it Remark.} We can further discuss the case where $\mu_{0} = 0$, $\mu_{j} = 0$ for $1\leqslant{j}\leqslant{\rho-1}$, and $r=0$. In this case,
the asymptotic log-concavity depends on whether $b_0$ is less than $1$. However, determining the explicit value of $b_0$ is non-trivial since it depends on both the recursion and initial values of the sequence $a_n$.

Note that when $a_n$ has the asymptotic form \eqref{t1}, $a_n/n^\alpha$ will have the asymptotic form 
\[
a_{n}\sim e^{Q(\rho,n)}\cdot n^{r-\alpha}\left(\sum_{s=0}^{M}b_{s}n^{-s/\rho} + o(n^{-M/\rho})\right). 
\]
Therefore, the above criteria applies to the sequence $\{\sqrt[n]{a_n/n^\alpha}\}_{n \ges 1}$ by replacing $r$ with $r-\alpha$.

With a similar discussion, we find a criterion which guarantees the asymptotic ratio log-convexity of the root sequence.
\begin{theorem}\label{t3}
Suppose that $\{a_n\}_{n\geqslant{1}}$ is a P-recursive sequence whose asymptotic expression is
\begin{align}
	a_{n}\sim  e^{Q(\rho,n)}\cdot n^{r}\left(\sum_{s=0}^{M}b_{s}n^{-s/\rho} + o(n^{-M/\rho})\right), \label{eq-1}   
\end{align}
where
\[
Q(\rho,n) = \mu_{0}n\log{n} + \sum_{j=1}^{\rho}\mu_{j}n^{j/\rho},
\]
with $\rho, M$ being positive integers and $\mu_{j}, r, b_{s}$ being real numbers such that 
\[
b_{0} >0 \quad \mbox{and} \quad M/\rho > 3.
\]
Then for the following three cases, the sequence $\{\sqrt[n]{a_n}\}_{n\geqslant{1}}$ is asymptotically ratio log-convex.
\begin{itemize}
	\item[{\rm (1)}]
	$\mu_{0}>{0}$.
	\item[{\rm (2)}]
	$\mu_{0} = 0$ and $ \mu_{k}<{0},$ where $k = \max\limits_{1\leqslant{j}\leqslant{\rho-1}} \{j \colon \mu_{j}\ne{0}\}$.
	\item[{\rm (3)}]
	$\mu_{0} = 0,\ \mu_{j} = 0$ for $1\leqslant{j}\leqslant{\rho-1}$, and $r<{0}$. 
\end{itemize}
\end{theorem}
\begin{proof}
Notice that     
\[
\left(\frac{\sqrt[n+1]{a_{n+1}}}{\sq}\right)^2 \leqslant \frac{\sqrt[n+2]{a_{n+2}}}{\sqrt[n+1]{a_{n+1}}} \cdot \frac{\sqrt[n]{a_{n}}}{\sqrt[n-1]{a_{n-1}}}
\]
if and only if
\[
\Delta' = -\frac{\log a_{n+2}}{n+2} + 3\frac{\log a_{n+1}}{n+1}-3\frac{\log a_n}{n}+\frac{\log a_{n-1}}{n-1} \leqslant 0.
\]
By analyzing each factor in the asymptotic expression of $a_n$, as done in the proof of Theorem~\ref{t1}, we derive that
 \begin{align*}
        \Delta' = &  -\frac{2 \mu_{0}}{n^3}-\sum\limits_{j=1}^{\rho-1} \frac{\mu_{j}(j/\rho-1)(j/\rho-2)(j/\rho-3)}{n^{4-j/\rho}}+\frac{r(6\log n -11)}{n^4}\\
        &  + \sum\limits_{s=0}^{M}\frac{\tilde{b}_{s}(1+s/\rho)(2+s/\rho)(3+s/\rho)}{n^{4+s/\rho}}+\dots+o\left(\frac{1}{n^{1+M/\rho}}\right).   
    \end{align*}
The proof follows immediately.
\end{proof}

\begin{exam}
\indent Let $G_{n}$ be the number of graphs on $[n]=\{1,\ldots,n\}$ whose every component is a cycle (see \cite[Example 3.7]{liu2018new}).  We have
\[G_{n+1}=(n+1) G_{n}- \binom{n}{2} G_{n-2}\]
with initial values $G_{0}=1, G_{1}=1, G_{2}=2.$ 

It can be computed by our {\tt Mathematica} package {\tt P-rec.m} that
\[
G_{n}
\sim C \cdot e^{n\log n - n}\left(1+\frac{11}{24n}+\frac{913}{1152n^2}+\frac{829543}{414720n^3}+o\left(\frac{1}{n^3}\right)\right).
\]
Since $\mu_0=1>0$, we immediately derive that the root sequence $\{\sqrt[n]{G_n}\}_{n \geqslant 1}$ is asymptotically log-concave and ratio log-convex. Moreover, the sequence $\{\sqrt[n]{G_n/n^{\alpha}}\}_{n \geqslant 1}$ is asymptotically log-concave for any real number $\alpha$.
\end{exam}

\section{Finding the explicit $N$}
In this section, we aim to  determine a specific value for $N$ such that $\{\sq\}_{n\geqslant{N}}$ is log-concave or ratio log-convex. In their work \cite{hou2019asymptotic}, Hou and Zhang introduced an algorithm for computing an integer $N$ along with lower and upper bounds for $r_n=a_{n+1}/a_n$ when $n \geqslant N$. Our approach relies on these bounds together with an upper bound for $a_n$. It is worth to note that deriving these bounds for $r_n$ requires the sequence  $\{a_n\}_{n \ges 1}$ to be bound preserving, as defined in  \cite{hou2019asymptotic}. 

Firstly, we give a criterion on the log-concavity of the root sequences. 
\begin{theorem}\label{3t1}
Let $\{a_n\}_{n\geqslant{1}}$ be a positive sequence. Suppose we can find an upper bound $h_n$ of $a_n$, a lower bound $f_n$ and an upper bound $g_n$ of  $\frac{a_{n+1}}{a_n}$, such that
\begin{equation}\label{31}
2 \log h_n + n(n+1) \log g_{n+1} - n(n+3) \log f_n \les 0, \quad \forall n \ges N.
\end{equation}
Then the root sequence $\{\sq\}_{n\geqslant{N}}$ is log-concave.
\end{theorem}

\begin{proof}
From the definition, the log-concavity of $\{\sq\}_{n\geqslant{N}}$ is equivalent to
\begin{equation}\label{eq-dif}
 \dfrac{\log {a_{n+2}}}{n+2} + \dfrac{\log {a_{n}}}{n} - \dfrac{2 \log {a_{n+1}}}{n+1} \leqslant{0}. 
\end{equation}
Denote $r_n =a_{n+1}/a_n$.  
Substituting $a_{n+2}=a_n r_n r_{n+1}$ and $a_{n+1}= a_n r_n$ in \eqref{eq-dif}, we get an equivalent inequality
\[
2 \log a_n + n(n+1) \log r_{n+1} - n (n+3) \log r_n \les 0.
\]
Notice that
\begin{align*}
&	2 \log a_n + n(n+1) \log r_{n+1} - n (n+3) \log r_n \\
\les \ & 2 \log h_n + n(n+1) \log g_{n+1} - n (n+3) \log f_n.
\end{align*}
The proof follows immediately. 
\end{proof}

\noindent
{\it Remark.} Assume that $\{a_n\}_{n \ges 1}$ is a bound preserving sequence. We can use the {\tt Mathematica} package {\tt P-rec.m} to compute $N$ and bounds $f_n,\ g_n$ for $a_{n+1}/a_n$. By using the command\\
\centerline{\tt RootLog[L, n, N, ini\_val, K]}\\
where 
\[L = N^d - ( R_{0}(n) + R_{1}(n)N + \cdots + R_{d-1}(n)N^{d-1} )\]
corresponds to the recurrence relation 
\[a_{n+d} = R_{0}(n)a_n + R_{1}(n)a_{n+1} + \cdots + R_{d-1}(n)a_{n+d-1},\]
{\tt ini\_val} is a list of initial values of the sequences and {\tt K} is the number of terms of the asymptotic ratio of the sequence. Then we can get
\[ \{\{f_{n},\ g_{n}\},\ N\} \]
which indicates that 
\[ f_{n} \leq a_{n+1}/a_{n} \leq g_{n}\]
for $n \geq N.$
And for the three cases in Theorem~\ref{t1}, we can always find an upper bound for $a_n$ such that \eqref{31} holds, implying the log-concavity of the root sequence. 

In fact, we may take 
\[
h_n = e^{Q(\rho,n)} n^{r+\epsilon}
\]
for a certain $\epsilon>0$.
On one hand, by aid of the upper bound $g_n$, we can show that $h_n$ is really an upper bounds for $a_n$. On the other hand, from \eqref{delta} we see that the leading term in the asymptotic expression of
\[
2 \log h_n + n(n+1) \log r_{n+1} - n (n+3) \log r_n 
\]
is negative, which ensures that \eqref{31} holds when $f_n,g_n$ are tight enough.

The following example illustrates the computation.
\begin{exam}
Let 
\[
a_n = f_n^{(5)} = \sum_{k=0}^n \binom{n}{k}^{5}
\]
be the Franel numbers of order $5$. The root sequence of $\{a_n\}_{n \ges 1}$ is log-concave.
\end{exam}

By~\cite{perlstadt1987some} or Zeilberger's algorithm, we find that $a_n$ satisfies the recurrence:
{\small
	\begin{multline*}
	32(55n^2+33n+6)(n-1)^4 a_{n-2}-(19415n^6 - 
	27181n^5 + 7453n^4 + 3289n^3 - 956n^2 \\
	- 276n + 96) a_{n-1}
	- (1155n^6 + 693n^5 - 732n^4 - 715n^3 + 45n^2 + 
	210n + 56) a_n \\
	+ (55n^2 - 77n + 28)(n + 1)^4 a_{n+1}=0,
\end{multline*}}\\[-5pt]
with initial values 
\[
f_0^{(5)}=1,\, f_1^{(5)}=2,\, f_2^{(5)}=34,\, f_3^{(5)}=488.
\] 	

By using the package {\tt P-rec.m}, we obtain the asymptotic expansion of $a_n$ 
\[
a_n \sim C\cdot 32^{n} n^{-2}\left(1-\frac{4}{5n}+ \frac{7}{25n} + \frac{2}{125n^3} + o\left(\frac{1}{n^3}\right)\right), 
\]
where $C$ is a certain constant. Furthermore, we obtain bounds
for $r_n = a_{n+1}/a_n$:
\[
f_n \le r_n \le g_n, \quad \forall\, n\geqslant{414},
\]
with
\[
f_n = 32-\frac{64}{n}+\frac{603}{5n^2}, \quad g_{n}=32-\frac{64}{n}+\frac{613}{5n^2}.
\]

Next, we will show that $h_n = 32^n/n$ is an upper bound for $a_n$ when $n \ges 414$ by induction.  Initially, one can check that the inequality holds for $n = 414$. Assume it holds for $a_n$, i.e., $a_n \le  h_n$. Notice that
\[
a_{n+1} = a_n r_n \le h_n g_n
\]
and
\[
\frac{h_{n+1}  - h_n g_n }{h_n}  = \frac{160n^2-293n-613}{5n^2(n+1)} > 0, \quad \forall\, n\ges 4.
\]
We thus derive that $a_{n+1} \le h_{n+1}$.

Now we will show that
\[
D(n) = 2 \log h_n + n(n+1) \log g_{n+1} - n (n+3) \log f_n \le 0, \quad n \ges 19.
\]
By {\tt Maple}, we compute that
\[
\lim_{n \to +\infty} D(n) = - \infty, \quad \lim_{n \to +\infty} D'(n) = \lim_{n \to +\infty} D''(n) = 0,
\]
and
{\small 
\begin{align*}
D'''(n) &= - (81527952168634716 + 155552087844136386 n + 10463853565472148 n^2 \\
		&+ 181717579054720476 n^3 - 30411335131752960 n^4 + 384356368351451520 n^5 \\
		&+ 197442488881497600 n^6 + 116628917999616000 n^7 + 79843811622912000 n^8 \\
		&+ 49427913441280000 n^9 - 24866003025920000 n^{10} + 7604118487040000 n^{11}\\
		&- 4945084416000000 n^{12} - 974336819200000 n^{13} + 67108864000000 n^{14})\\
		&/(n^3 (1 + n)^2 (453 + 160 n^2)^3 (603 - 320 n + 160 n^2)^3).
\end{align*}
}\\
By calculating the maximum root of the numerator in the expression for $D'''(n)$, we determine that $D'''(n)<0$ when $n \geqslant 19$. Consequently, $D''(n)>0$ and $D'(n)<0$ for $n \geqslant 19$.  It is straightforward to check that $D(19)<0$, implying that $D(n)<0$ for $n \geqslant 19$. By examining the initial terms, we  derive that $D_{n}<0$ for $n \geqslant 16$.

In summary, we have proved that the root sequence of $\{f_n^{(5)}\}_{n \ges 414}$ is log-concave. By further examination of the initial values, we see that the root sequence of $\{f_n^{(5)}\}_{n\geqslant{1}}$ is log-concave, completing the proof.

We have implement a \texttt{Mathematica} package {\tt P-rec.m} (which is available at \cite{pre}) to do the above computations. Utilizing this package, we confirm Conjecture 3.10 in~\cite{sun2013conjectures} for $r=3, 4, 5$. We also reprove the log-concavity of the root sequences of the following sequences: $\{S_{n}\}_{n\geqslant{1}}$ in~\cite{sun2016proof},
$\{{R_{n}}\}_{n\geqslant{5}}$ in~\cite{sun2018ratio}, $\{{a_{n}}\}_{n\geqslant{2}}$ and $\{{b_{n}}\}_{n\geqslant{1}}$ in~\cite{zhang2021proof},  the Catalan-Larcombe-French sequence $\{{P_{n}}\}_{n\geqslant{1}}$ in~\cite{zhao2016sun}, the Zagier numbers $\{{Z_n}\}_{n \geqslant 1}$, the Ap\'{e}ry numbers $\{{A_n}\}_{n \geqslant 1}$, the Domb numbers $\{{D_n}\}_{n \geqslant 1}$, the Motzkin numbers $\{{M_{n}}\}_{n \geqslant 1}$, the Cohen-Rhin numbers $\{{U_n}\}_{n \geqslant 1}$ in~\cite{xia2018log}, and the sequences in~\cite{chen2014infinitely} and~\cite{luca2012some}.

Replacing $a_n$ by $a_n/n^\alpha$, we can find the explicit $N_{\alpha}$ such that the root sequence of $\{{a_n / n^{\alpha}}\}_{n \geqslant N_{\alpha}}$ is log-concave. Here is an example.

\begin{exam}
Let $C_{n} = {2n \choose n}/(n+1)$ be the Catalan numbers and  $a_n=1/C_n$ (see \cite{hou2021log}). Then the root sequence of $\{{a_n / n^{2}}\}_{n \geqslant 6}$ is log-concave.
\end{exam}

Let $b_n=a_n/n^2$. It is straightforward to see that
\[
(4n+2)(n+1)^2b_{n+1}-n^2(n+2)b_{n}=0.
\]
By the package {\tt P-rec.m}, we find that
\[
b_{n}\sim C\cdot n^{-\frac{1}{2}}4^{-n} \left(1+\frac{9}{8n}+\frac{17}{128n^2}+\frac{3}{1024n^3}+o\left(\frac{1}{n^3}\right)\right),\]
for some constant $C$. Noting that
\[
r_n = \frac{b_{n+1}}{b_n} = \frac{n^2(n+2)}{2(2n+1)(n+1)^{2}},
\]
we may take $f_n=g_n=r_n$ as the lower and upper bounds for $r_n$.
Moreover, for $n \geqslant 2$, we have an upper bound for $b_n$:
\[h_n=\frac{4^{-n}}{n^{{1}/{4}}}.\]
This can be shown by noting 
\[
\lim_{n \to +\infty} D(n) = 0, \quad D'(n) = -\frac{2n^2+23n+14}{4n(2n^3+7n^2+7n+2)} < 0,
\]
where
\[
D(n) = \log(h_{n+1}) - \log(g_n h_n).
\]
Similarly, we can verify that  
\[
D(n) = 2 \log h_n + n(n+1) \log g_{n+1} - n (n+3) \log f_n \le 0, \quad n \ges 26.
\]
Finally, by checking initial values, we find that $\{\sqrt[n]{b_n}\}_{n \geqslant 20}$ is log-concave.

Similar to Theorem~\ref{3t1}, we can establish a criterion for the ratio log-convexity of the root sequence $\{\sq\}_{n \geqslant N}$. We will omit the proof as it closely parallels that of Theorem~\ref{3t1}. 
\begin{theorem}\label{t-rN}
Let $\{a_n\}_{n\geqslant{1}}$ be a positive sequence. Suppose we can find an upper bound $h_n$ of $a_n$, a lower bound $f_n$ and an upper bound $g_n$ of  $\frac{a_{n+1}}{a_n}$, such that for $n \ges N$
\[
	6 \log h_n  + (n^2-n)(2n+5) \log g_n - (n^2+n)(n+2) \log f_{n-1} - (n^3-n) \log f_{n+1}  \les 0.
\]
Then the root sequence $\{\sq\}_{n\geqslant{N}}$ is ratio log-convex.
\end{theorem}

We give an example to illustrate the application of Theorem~\ref{t-rN}.
\begin{exam}
Let $M_{n}$ be the Motzkin numbers which satisfy 
\[(n + 4)M_{n+2} - (2n + 5)M_{n+1} - 3(n + 1)M_n = 0, \quad n\geqslant 0.\]
Then the root sequence $\{\sqrt[n]{M_n}\}_{n\geqslant 1}$ is ratio log-convex.
\end{exam}

By the package {\tt P-rec.m}, we find that
\[M_{n} \sim C\cdot 3^{n} n^{-3/2} \left(1-\frac{39}{16n}+\frac{2665}{512n^2}-\frac{87885}{8192n^3}+o\left(\frac{1}{n^3}\right)\right),\]
for some constant $C.$ Moreover, for $n\geqslant{228}$ we have the lower and upper bound of $r_n = M_{n+1}/M_n$: 
\[f_n = 3-\frac{9}{2n}+\frac{207}{16n^2}-\frac{157}{4n^3},\]
and 
\[g_n = 3-\frac{9}{2n}+\frac{207}{16n^2}-\frac{149}{4n^3}.\]

Now we take 
\[h_{n}=\frac{3^{n}}{n}.\] 
It can be verified that for $n\geqslant{15},$ 
\[
   	6 \log h_n  + (n^2-n)(2n+5) \log g_n - (n^2+n)(n+2) \log f_{n-1} - (n^3-n) \log f_{n+1}  \les 0.
\]
Finally, by checking initial values, we derive that $\{\sqrt[n]{M_n}\}_{n\geqslant{1}}$ is ratio log-convex. 

In a similar way, we show that the root sequences of the following combinatorial sequences are all ratio log-convex: 
Fine numbers $\{{f_n}\}_{n\geqslant 4}$, the central Delannoy numbers $\{{D_{n}}\}_{n\geqslant 1}$, the Domb numbers $\{{D_{n}}\}_{n\geqslant 1}$, the numbers of tree-like polyhexes with $n+1$ hexagons $\{{t_{n}}\}_{n\geqslant{1}}$, and the Catalan-Larcombe-French sequence $\{{P_{n}}\}_{n\geqslant 1}$.


\end{document}